\documentclass[draft,a4,reqno,11pt]{amsart}
\usepackage{verbatim}
\usepackage{amssymb,amsmath,amsthm}
\usepackage{amsfonts}
\usepackage{array}

\newcommand{\qbin}[2]{\genfrac{[}{]}{0pt}{}{#1}{#2}_{\zeta_n}}
\newcommand{\qqbin}[2]{\genfrac{[}{]}{0pt}{}{#1}{#2}_q}

\newtheorem{theorem}{Theorem}[section]
\newtheorem{lemma}{Lemma}[section]
\newtheorem{corollary}{Corollary}[section]
\newtheorem{conjecture}{Conjecture}[section]

\begin{document}

\title{On $3-2-1$ values of finite multiple harmonic $q$-series at roots of unity}

\author[\tiny Kh.~Hessami Pilehrood]{Kh.~Hessami Pilehrood}
\address{The Fields Institute for Research in Mathematical Sciences, 222 College St, Toronto, Ontario M5T 3J1 Canada}
\email{hessamik@gmail.com}

\author{T.~Hessami Pilehrood}
\address{The Fields Institute for Research in Mathematical Sciences, 222 College St, Toronto, Ontario M5T 3J1 Canada}
\email{hessamit@gmail.com}

\author{R.~Tauraso}
\address{Dipartimento di Matematica, % \\
Universit\`a di Roma ``Tor Vergata'', % \\
via della Ricerca Scientifica, %\\
00133 Roma, Italy}
\email{tauraso@mat.uniroma2.it}


\begin{abstract} We mainly answer two open questions about finite multiple harmonic $q$-series on 3-2-1 indices at roots of unity, posed recently by H.~Bachmann, Y. Takeyama, and K. Tasaka. 
Two conjectures regarding cyclic sums which generalize the given results are also provided.
\end{abstract}

\maketitle

\section{Introduction.}

For two $r$-tuples of non-negative integers ${\bf s} = (s_1, \ldots, s_r)$ and  ${\bf t} = (t_1, \ldots, t_r)$ and a positive integer $n$, with a complex number $q$
satisfying $q^m\ne 1$ for $n>m>0$, we define two classes of multiple $q$-harmonic sums
\begin{align*}
H_n({\bf s}; {\bf t}; q)=H_n(s_1, \ldots, s_r; t_1, \ldots, t_r;q) & = 
\sum_{1\le k_1<\cdots<k_r\le n}\frac{q^{k_1t_1+\cdots+k_rt_r}}{[k_1]_q^{s_1}\cdots[k_r]_q^{s_r}}, \\
H_n^\star({\bf s}; {\bf t}; q)=H_n^\star(s_1, \ldots, s_r; t_1, \ldots, t_r;q) & = 
\sum_{1\le k_1\le\cdots\le k_r\le n}\frac{q^{k_1t_1+\cdots+k_rt_r}}{[k_1]_q^{s_1}\cdots[k_r]_q^{s_r}},
\end{align*}
where 
$$
[n]_q=\frac{1-q^n}{1-q}=1=q+\cdots+q^{n-1}
$$
is a $q$-analog of positive integer $n$. By convention, we put $H_n(\emptyset)=H_n^\star(\emptyset)=1$, and $H_n({\bf s}; {\bf t}; q)=0$ if $n<r$.
The number $w({\bf s})=\sum_{j=1}^rs_j$ is called the weight of the multiple harmonic sum.

 For a primitive $n$-th root of unity $\zeta_n$, following  work \cite{BTT:20}, we adopt the  notation %(note that we reversed the order of summation for convenience in our settings)
\begin{align*}
z_n({\bf s}; \zeta_n)&=
H_{n-1}({\bf s}; {\bf s}-\{1\}^r; \zeta_n), \\
z_n^\star({\bf s}; \zeta_n)&= 
H_{n-1}^\star({\bf s}; {\bf s}-\{1\}^r; \zeta_n),
\end{align*}
where 
%$\overline{\bf s}=(s_r, s_{r-1}, \ldots, s_1)$ stands for the reverse of ${\bf s}=(s_1, \ldots, s_{r-1}, s_r)$
$\{a\}^r$ denotes an $r$-tuple with $r$ consecutive copies of the letter $a$ (note that we reversed the order of summation for convenience in our settings).

In \cite{BTT:20}, Bachmann, Takeyama, and Tasaka studied special values of $z_n(\{k\}^r; \zeta_n)$ and in particular for $k=1, 2, 3$, showed that
\begin{align} 
z_n(\{1\}^r; \zeta_n)&=\frac{1}{n}\binom{n}{r+1}(1-\zeta_n)^r,  \label{1} \\
z_n(\{2\}^r; \zeta_n)&=\frac{(-1)^r}{n(r+1)}\binom{n+r}{2r+1}(1-\zeta_n)^{2r},  \label{2} \\
z_n(\{3\}^r; \zeta_n)&=\frac{1}{n^2(r+1)}\left(\binom{n+2r+1}{3r+2}+(-1)^r\binom{n+r}{3r+2}\right)(1-\zeta_n)^{3r}.  \label{3}
\end{align}
The authors of \cite{BTT:20} also formulated two open questions for finite multiple harmonic $q$-series $z_n$ on 3-2-1 indices, namely,
\begin{equation} \label{4}
\begin{split}
z_n(\{1\}^a, 2, \{1\}^b; \zeta_n)+z_n(\{1\}^b,2,\{1\}^a; \zeta_n) &\stackrel{?}{=}-\frac{1}{n}\binom{n+1}{a+b+3}(1-\zeta_n)^{a+b+2}, \\[3pt]
z_n(\{2\}^a, 3, \{2\}^b; \zeta_n)+z_n(\{2\}^b,3,\{2\}^a; \zeta_n) &\stackrel{?}{=} \frac{(-1)^{a+b}}{n(a+b+2)}\binom{n+a+b+1}{2(a+b)+3}(1-\zeta_n)^{2(a+b)+3}.
\end{split}
\end{equation}
In this paper,  we prove the above relations and obtain related formulas for corresponding values of $\xi({\bf s})$, which are defined as the limit values (see \cite[Thm.~1.2]{BTT:18})
$$
\xi({\bf s})=\lim_{n\to\infty}z_n({\bf s}; e^{\frac{2\pi i}{n}}).
$$
Note that when $n$ is a prime, formulas (\ref{4}) as well as (\ref{2}) and (\ref{3})  follow from our results on $q$-congruences for multiple $q$-harmonic sums
\cite[Thm.~4.1, Thm.~5.1, Thm.~6.1, and Thm.~8.3]{HPT}, while formula (\ref{1}) follows from \cite[Cor.~2.2]{Zh13}.  The methods of our paper \cite{HPT} can be easily adjusted to prove (\ref{4}) for arbitrary positive integer $n$.

\begin{theorem} \label{T1}
For all non-negative integers $a, b$ and any $n$-th primitive root of unity $\zeta_n$, 
\begin{equation*}
z_n(\{2\}^a, 3, \{2\}^b; \zeta_n)+z_n(\{2\}^b,3,\{2\}^a; \zeta_n) =\frac{(-1)^{a+b}}{n(a+b+2)}\binom{n+a+b+1}{2(a+b)+3}(1-\zeta_n)^{2(a+b)+3}.
\end{equation*}
\end{theorem}

\begin{theorem} \label{T2}
For all non-negative integers $a, b$ and any $n$-th primitive root of unity $\zeta_n$, 
\begin{equation*}
z_n(\{1\}^a, 2, \{1\}^b; \zeta_n)+z_n(\{1\}^b,2,\{1\}^a; \zeta_n) = -\frac{1}{n}\binom{n+1}{a+b+3}(1-\zeta_n)^{a+b+2}.
\end{equation*}
\end{theorem}

\noindent The complex numbers $\xi({\bf s})$ are of interest in view of their connections to the finite and symmetric multiple zeta values as was shown in \cite{BTT:18}.
After letting $\zeta_n=e^{\frac{2\pi i}{n}}$ in Theorem~\ref{T1} and Theorem~\ref{T2}, and by noting that for $j,k\in\mathbb{N}$,
$$\lim_{n\to\infty}\binom{n+j}{k}\frac{k!}{n^k}=1,\quad\text{and}\quad\lim_{n\to\infty}n(1-e^{\frac{2\pi i}{n}})=-2\pi i,$$
we obtain the following corollary.

\begin{corollary} \label{C1}
For all non-negative integers $a, b$,
\begin{equation*}
\xi(\{1\}^a, 2, \{1\}^b)+\xi(\{1\}^b,2,\{1\}^a) = -\frac{(-2\pi i)^{a+b+2}}{(a+b+3)!},
\end{equation*}
and
\begin{equation*}
\xi(\{2\}^a, 3, \{2\}^b)+\xi(\{2\}^b,3,\{2\}^a) 
=0.
\end{equation*}
\end{corollary}
\noindent  Note that the last relation can also be readily obtained from the definition of the symmetric multiple zeta values (see, for example, \cite[Def.~2.5]{BTT:18}).

Finally, we put forward the following conjectures regarding cyclic sums of multiple $q$-harmonic sums $z_n$ at roots of unity, which generalize both of the theorems above.

\begin{conjecture}[Cyclic-sum]
Let $d_0,d_1,\dots,d_{t}$ be non-negative integers. Then

\begin{itemize}

\item[(i)] For every integer $n>r$, where $r=\sum_{j=0}^{t}d_j+2t$, and any primitive root of unity $\zeta_n$,
\begin{equation*}
\sum_{j=0}^{t}z_n\left(
\{1\}^{d_j},2,\{1\}^{d_{j+1}},2,\dots,2,\{1\}^{d_{j+t}}
\right)=
\frac{(-1)^t}{n}\binom{n+t}{r+1}(1-\zeta_n)^{r}.
\end{equation*}

\item[(ii)] For every integer $n>r$, where $r=\sum_{j=0}^{t}2d_j+3t$, and any primitive root of unity $\zeta_n$,
\begin{equation*}
\sum_{j=0}^{t}z_n\left(
\{2\}^{d_j},3,\{2\}^{d_{j+1}},3,\dots,3,\{2\}^{d_{j+t}}
\right)\in (1-\zeta_n)^r{\mathbb Q}.
\end{equation*}
\end{itemize}
In both sums above it is understood that $d_j=d_k$ if $j\equiv k$ modulo $t+1$.
\end{conjecture}
Note that the case $t=1$ follows from Theorem \ref{T1} and Theorem \ref{T2}. The case of arbitrary $t$ when all $d_j$ are zeros follows from (\ref{2}) and (\ref{3}).

\section{Proof of Theorem \ref{T1}.}

Let $\overline{\bf s}=(s_r, s_{r-1}, \ldots, s_1)$ denote the reverse of ${\bf s}=(s_1, \ldots, s_{r-1}, s_r)$. Then we have the following relations.

\begin{lemma} \label{L1}
Let ${\bf s}=(s_1, \ldots, s_r)$ and ${\bf t}=(t_1, \ldots, t_r)$ be two $r$-tuples of non-negative integers, and $\zeta_n$ be an $n$-th primitive root of unity. Then
\begin{align*}
H_{n-1}({\bf s}; {\bf t}; \zeta_n) &=(-1)^{w({\bf s})}H_{n-1}(\overline{\bf s}; \overline{\bf s}-\overline{\bf t}; \zeta_n), \\
H_{n-1}^\star({\bf s}; {\bf t}; \zeta_n) &=(-1)^{w({\bf s})}H_{n-1}^\star(\overline{\bf s}; \overline{\bf s}-\overline{\bf t}; \zeta_n), 
\end{align*}
and in particular,
\begin{align}
z_{n}({\bf s}; \zeta_n) &=
(-1)^{w({\bf s})}H_{n-1}(\overline{\bf s}; \{1\}^r; \zeta_n), \label{5}\\
z_{n}^\star({\bf s}; \zeta_n) &=
(-1)^{w({\bf s})}H_{n-1}^\star(\overline{\bf s}; \{1\}^r; \zeta_n). \nonumber
\end{align}
\end{lemma}
\begin{proof}
Replacing each $k_i$ by $n-k_i$ and reversing the order of summation, we get
\begin{equation*}
\begin{split}
H_{n-1}({\bf s}; {\bf t}; \zeta_n)&=\sum_{0<k_1<\cdots<k_r<n}\frac{\zeta_n^{t_1k_1+\ldots+t_rk_r}}{[k_1]_{\zeta_n}^{s_1}\dots [k_r]_{\zeta_n}^{s_r}}
=\sum_{0<n-k_1<\cdots<n-k_r<n}\frac{\zeta_n^{t_1(n-k_1)+\ldots+t_r(n-k_r)}}{[n-k_1]_{\zeta_n}^{s_1}\dots [n-k_r]_{\zeta_n}^{s_r}} \\
&=\sum_{0<k_r<k_{r-1}\cdots<k_1<n}\frac{\zeta_n^{-t_1k_1-\ldots-t_rk_r}}{[k_1]_{\zeta_n}^{s_1}\dots [k_r]_{\zeta_n}^{s_r}}\times (-1)^{w({\bf s})} \zeta_n^{k_1s_1+\ldots + k_rs_r} \\[3pt]
&=(-1)^{w({\bf s})}H_{n-1}(\overline{\bf s}; \overline{\bf s}-\overline{\bf t}; \zeta_n),
\end{split}
\end{equation*}
where we used  the identity
\begin{equation}
\label{6}
[n-k_i]_{\zeta_n}=\frac{1-\zeta_n^{n-k_i}}{1-\zeta_n}=\frac{1-\zeta_n^{-k_i}}{1-\zeta_n}=-\zeta_n^{-k_i}[k_i]_{\zeta_n}.
\end{equation}
Setting ${\bf t}={\bf s}-\{1\}^r$, we get (\ref{5}). The proofs for the multiple harmonic star sums are similar.
\end{proof}

\noindent{\bf Proof of Theorem \ref{T1}.}
We have
\begin{equation*}
\begin{split}
& \qquad z_n(\{2\}^a, 3, \{2\}^b; \zeta_n)+(1-\zeta_n)z_n(\{2\}^{a+b+1}) \\[3pt]
&=\sum_{0<k_1<\cdots<k_a}\frac{\zeta_n^{k_1+\ldots+k_a}}{[k_1]_{\zeta_n}^2\cdots [k_a]_{\zeta_n}^2}\sum_{k_a<k_{a+1}<k_{a+2}}\left(\frac{\zeta_n^{2k_{a+1}}}{[k_{a+1}]_{\zeta_n}^3}+
\frac{(1-\zeta_n)\zeta_n^{k_{a+1}}}{[k_{a+1}]_{\zeta_n}^2}\right) \\[3pt]
&\times\sum_{k_{a+1}<k_{a+2}<\cdots<k_{a+b+1}\leq n}\frac{\zeta_n^{k_{a+2}+\ldots+k_{a+b+1}}}{[k_{a+2}]_{\zeta_n}^2\cdots [k_{a+b+1}]_{\zeta_n}^2} \\[3pt]
&=H_{n-1}(\{2\}^a, 3, \{2\}^b; \{1\}^{a+b+1}; \zeta_n)=-z_n(\{2\}^b, 3, \{2\}^a; \zeta_n),
\end{split}
\end{equation*}
where in the last equality we used (\ref{5}). Hence
\begin{equation*}
z_n(\{2\}^a, 3, \{2\}^b; \zeta_n)+z_n(\{2\}^b, 3, \{2\}^a; \zeta_n)=-(1-\zeta_n)z_n(\{2\}^{a+b+1}),
\end{equation*}
which, by (\ref{2}), implies the theorem. \qed

\section{Proof of Theorem \ref{T2}.}

The $q$-binomial coefficient, or Gaussian coefficient, when $q$ is specified to a primitive root of unity has the following properties.

\begin{lemma} \label{L2}
Let $n>1$ be a positive integer. Then for any primitive $n$-th root of unity $\zeta_n$ and $1\le k <n$,
\begin{equation*}
\qbin{n-1}{k}=(-1)^k\zeta_n^{-\binom{k+1}{2}}.
\end{equation*}
\end{lemma}
\begin{proof} We have
\begin{equation*}
\qbin{n-1}{k}=\prod_{j=1}^k\frac{[n-j]_{\zeta_n}}{[j]_{\zeta_n}}=\prod_{j=1}^k\frac{1-\zeta_n^{n-j}}{1-\zeta_n^j}=\prod_{j=1}^k\frac{1-\zeta_n^{-j}}{1-\zeta_n^j}
=\prod_{j=1}^k(-\zeta_n^{-j})=(-1)^k\zeta_n^{-\binom{k+1}{2}}.
\end{equation*}
\end{proof}
\noindent The proof of Theorem~\ref{T2} is based on the following multiple $q$-binomial identity.

\vspace{0.1cm}

\noindent {\bf Theorem A\,}{\rm (}\cite{HPT}, Thm.~8.1{\rm )}
{\it
Let $n, s_1, \ldots, s_r$ be positive integers. Then 
\begin{equation*}
\sum_{k=1}^n \qqbin{n}{k}(-1)^kq^{\binom{k+1}{2}}\sum_{1\leq k_1<k_2<\ldots<k_r=k}\,\prod_{i=1}^r\frac{q^{(s_i-1)k_i}}{[k_i]_q^{s_i}}
=(-1)^r\underset{j_i<j_{i+1}, \, i\in I} {\sum_{1\le j_1\le j_2\le\ldots\le j_w\le n}} \,\prod_{i=1}^w\frac{q^{j_i}}{[j_i]_q},
\end{equation*}
where $w=w({\bf s})=\sum_{i=1}^r s_i$, $I=\{s_1, s_1+s_2,  \ldots, s_1+s_2+\dots+ s_{r-1}\}$, and the sum on the right is taken over all integers
$j_1, \ldots, j_w$ satisfying the conditions $1\le j_i\le n$, $j_i<j_{i+1}$ for $i\in I$, and $j_i\le j_{i+1}$ otherwise.
}

From Theorem A we get a kind of duality for finite multiple $q$-harmonic sums $z_n$ at roots of unity. 

\begin{theorem} \label{T3}
Let $n, s_1, \ldots, s_r$ be positive integers. Then 
\begin{equation} \label{7}
z_n({\bf s}; \zeta_n)=(-1)^r\underset{j_i<j_{i+1}, \, i\in I} {\sum_{1\leq j_1\le j_2\le\ldots\le j_w<n}} \,\prod_{i=1}^w\frac{\zeta_n^{j_i}}{[j_i]_{\zeta_n}},
\end{equation}
where $w=w({\bf s})=\sum_{i=1}^r s_i$, $I=\{s_1, s_1+s_2,  \ldots, s_1+s_2+\dots+ s_{r-1}\}$, and the sum on the right is taken over all integers
$j_1, \ldots, j_w$ satisfying the conditions $1\le j_i\le n$, $j_i<j_{i+1}$ for $i\in I$, and $j_i\le j_{i+1}$ otherwise.
\end{theorem}
\begin{proof}
To get (\ref{7}), we replace $n$ by $n-1$, and $q$ by a primitive root of unity $\zeta_n$ in Theorem A, and apply Lemma \ref{L2}.
\end{proof}

\noindent {\bf Proof of Theorem \ref{T2}.} Let ${\bf s}=(\{1\}^a, 2, \{1\}^b)$ and $w=w({\bf s})=a+b+2$. Applying Theorem~\ref{T3} and noticing that 
$I=\{1,2,\ldots, a, a+2, a+3, \ldots, a+b+1\}$, we get
\begin{equation*}
\begin{split}
-z_n(\{1\}^a,  \, & 2, \{1\}^b; \zeta_n)  =(-1)^{a+b}\sum_{1\leq j_1<j_2<\ldots<j_{a+1}\le j_{a+2}<j_{a+3}<\ldots <j_w< n}\, \prod_{i=1}^w\frac{\zeta_n^{j_i}}{[j_i]_{\zeta_n}} \\
& =(-1)^{a+b}\sum_{0<n-j_1<\ldots<n-j_{a+1}\le n-j_{a+2}<n-j_{a+3}<\ldots <n-j_w< n}\, \prod_{i=1}^w\frac{\zeta_n^{n-j_i}}{[n-j_i]_{\zeta_n}}.
\end{split}
\end{equation*}
Applying identity (\ref{6}), we get
\begin{equation*}
\begin{split}
-z_n(&\{1\}^a, 2, \{1\}^b; \zeta_n)  =\sum_{n> j_1 >\ldots>j_{a+1}\ge j_{a+2}>j_{a+3}>\ldots >j_w\ge 1}
\frac{1}{[j_1]_{\zeta_n}\cdots [j_w]_{\zeta_n}} \\
& = \sum_{n> j_1 >\ldots>j_{a+1}>j_{a+3}>\ldots >j_w\ge  1} 
\frac{1}{[j_1]_{\zeta_n}\cdots [j_a]_{\zeta_n}[j_{a+1}]_{\zeta_n}^2[j_{a+3}]_{\zeta_n}\cdots [j_w]_{\zeta_n}} +
z_n(\{1\}^w; \zeta_n).
\end{split}
\end{equation*}
Noticing that
\begin{equation*}
\frac{1}{[j_{a+1}]_{\zeta_n}^2}=\frac{(1-\zeta_n)[j_{a+1}]_{\zeta_n}+\zeta_n^{j_{a+1}}}{[j_{a+1}]_{\zeta_n}^2}=\frac{1-\zeta_n}{[j_{a+1}]_{\zeta_n}}+\frac{\zeta_n^{j_{a+1}}}{[j_{a+1}]_{\zeta_n}^2},
\end{equation*}
we obtain
\begin{equation*}
-z_n(\{1\}^a, 2, \{1\}^b; \zeta_n)  =(1-\zeta_n)z_n(\{1\}^{w-1}; \zeta_n)+z_n(\{1\}^b, 2, \{1\}^a; \zeta_n)+z_n(\{1\}^w; \zeta_n).
\end{equation*}
Therefore, by (\ref{1}),
\begin{equation*}
\begin{split}
z_n(\{1\}^a, 2, \{1\}^b; \zeta_n)&+z_n(\{1\}^b, 2, \{1\}^a; \zeta_n) =-(1-\zeta_n)z_n(\{1\}^{w-1}; \zeta_n)-z_n(\{1\}^w; \zeta_n) \\
&=-\frac{1}{n}(1-\zeta_n)\binom{n}{w}(1-\zeta_n)^{w-1}-\frac{1}{n}\binom{n}{w+1}(1-\zeta_n)^w \\
&=-\frac{1}{n}\binom{n+1}{w+1}(1-\zeta_n)^w.
\end{split}
\end{equation*}
\qed

\end{document}